\documentclass[english]{amsart}

\usepackage{mathrsfs}
\usepackage{mathtools}
\usepackage{graphicx}
\usepackage{array}
\usepackage{amsmath}
\usepackage{amssymb}
\usepackage{amsthm}
\usepackage{amsfonts}
\usepackage{xcolor}
\usepackage[pagebackref,hyperindex=true,colorlinks=true,linkcolor=blue,citecolor=biblio,urlcolor=biblio]{hyperref}

\usepackage{cleveref}
\crefname{subsection}{subsection}{subsections}

\usepackage{tikz,tkz-tab}
\usetikzlibrary{matrix,cd,arrows.meta,fit,calc,positioning,decorations.pathreplacing,decorations.pathmorphing}
\tikzcdset{
arrow style=tikz,
diagrams={>={Stealth[scale=0.8]}}
}

\usepackage[shortlabels]{enumitem}
\setlist[enumerate]{label=\rm{(\arabic*)}}

\theoremstyle{plain}
\newtheorem{thm}{Theorem}[section]
\newtheorem{thmIntro}{Theorem}

\newtheorem*{prpIntro2}{\Cref{prpMMGluedMap}}

\newtheorem{prpIntro}[thmIntro]{Proposition}

\newtheorem{prp}[thm]{Proposition}
\newtheorem{lmm}[thm]{Lemma}

\newtheorem*{algo*}{Algorithm}

\theoremstyle{definition}
\newtheorem{ntt}[thmIntro]{Notation}

\newtheorem{prpdfn}[thm]{Proposition-Definition}
\newtheorem{ex}[thm]{Example}

\newtheorem{dfn}[thm]{Definition}

\newtheorem{rmk}[thm]{Remark}

\numberwithin{equation}{section}

\newtheorem{cjc}[thm]{Conjecture}

\DeclareMathOperator{\codim}{codim}

\DeclareMathOperator{\lgth}{length}

\DeclareMathOperator{\Proj}{Proj}
\DeclareMathOperator{\BiProj}{BiProj}

\DeclareMathOperator{\Tor}{Tor}
\newcommand*{\mrm}[1]{\mathrm{#1}}

\newcommand*{\mbf}[1]{\mathbf{#1}}
\newcommand*{\mc}[1]{\mathcal{#1}}

\newcommand*{\p}[1]{\mathbb{P}^{#1}}
\newcommand*{\kk}{\mrm{k}}
\newcommand*{\RR}{\mrm{R}}
\renewcommand*{\SS}{\mrm{S}}
\newcommand*{\BB}{\mrm{B}}
\newcommand*{\R}{\mathcal{R}}
\DeclareMathOperator{\I}{I}
\DeclareMathOperator{\J}{J}
\DeclareMathOperator{\Hilb}{H}
\newcommand*{\A}{\mathcal{A}}
\DeclareMathOperator{\Num}{Num}
\newcommand*{\G}{\Gamma}
\DeclareMathOperator{\Sym}{Sym}
\newcommand*{\X}{\mathbb{X}}
\newcommand*{\V}{\mathbb{V}}

\newcommand*{\gm}[2]{[#1|#2]}
\newcommand*{\msf}[1]{\mathsf{#1}}

\usepackage{xparse}
\newcommand{\e}[2]{\deg^{#2}_{\mathbb{P}}{#1}}
\DeclareMathOperator{\Conv}{Conv}
\renewcommand{\d}[2]{d_{#2}(#1)}
\DeclareMathOperator{\NP}{NP}
\DeclareMathOperator{\Vol}{Vol}
\DeclareMathOperator{\MV}{MV}

\date{\today}
\title[Gluing determinantal Cremona maps]{Gluing determinantal Cremona maps}

\author{R\'emi Bignalet-Cazalet}
\address{Universit\`a degli studi di Genova, Dipartimento di matematica,
Via Dodecaneso 35, 16146 Genova (GE), Italy }
\email{bignalet@dima.unige.fr}
\thanks{The author was funded by the European Union's Horizon 2020 research and innovation program as a  Marie Sk\l{}odowska-Curie fellow of the Istituto Nazionale di Alta Matematica "Francesco Severi" grant No. 713485.
}


\begin{document}
\definecolor{biblio}{rgb}{0,0.65,1}
\definecolor{indigo}{rgb}{0.29,0,0.51}

\definecolor{zzttqq}{rgb}{0.6,0.2,0}
\definecolor{qqqqff}{rgb}{0,0,1}
\definecolor{xdxdff}{rgb}{0.49,0.49,1}
\definecolor{cqcqcq}{rgb}{0.75,0.75,0.75}

\begin{abstract}
We study determinantal Cremona maps, i.e.\ birational maps whose base ideal is the maximal minors ideal of a given matrix $\Phi$, via the resolution of the polynomials systems defined by $\Phi$. Using convex geometry, this approach leads in particular to describe the projective degrees of some almost linear determinantal maps.
\end{abstract}
\maketitle
%\tableofcontents
\section*{Introduction}\label{labIntro}
In this note, we consider rational maps $f=(f_0:\ldots:f_n):\p{n}_\kk\dashrightarrow\p{n}_\kk$ defined by $n+1$ homogeneous polynomials $f_0,\ldots,f_n\in\RR=\kk[x_0,\ldots,x_n]$ without common factor, that are determinantal i.e.\ that the \emph{base ideal} $\I_f=(f_0,\ldots,f_n)\subset \RR$ of $f$ is the $n$-minors ideal of its minimal presenting matrix $\Phi_f$.

Our guiding remark is that the so-called \emph{standard Cremona map} (see \cite{PanGonzalez2005CarClassesDetCremTrans} for more references about the standard Cremona maps)\[\tau_n=(x_1\cdots x_n:x_0x_2\cdots x_n:\ldots:x_0\cdots x_{n-1}):\p{n}_\kk\dashrightarrow\p{n}_\kk\] can be interpreted as the gluing of $n$ standard Cremona maps $\tau^{(1)}_1=(x_1:x_0):\p{1}_\kk\dashrightarrow \p{1}_\kk$, $\tau^{(2)}_1=(x_2:x_1):\p{1}_\kk\dashrightarrow \p{1}_\kk$, $\ldots$, $\tau^{(n)}_1=(x_{n}:x_{n-1}):\p{1}_\kk\dashrightarrow \p{1}_\kk$. Indeed, $\I_{\tau_n}$ is the $n$-minors ideal of the $(n+1)\times n$ matrix 
\[\Phi_{\tau_n}=\begin{scriptsize}\begin{pmatrix}
   x_0  &    0   & \ldots & \ldots & 0      \\
   -x_1 &   x_1  & \ddots &        & \vdots \\
    0   &  -x_2  & \ddots & \ddots & \vdots \\
 \vdots & \ddots & \ddots & \ddots &    0   \\
 \vdots &        & \ddots & \ddots & x_{n-1}\\
    0   & \ldots & \ldots &    0   & -x_{n}
\end{pmatrix}\end{scriptsize}
\] which is the concatenation of the presenting matrices $\Phi_{\tau^{(1)}_1}=\begin{scriptsize}
\begin{pmatrix}
x_0 \\ -x_1
\end{pmatrix}
\end{scriptsize},\ldots,\Phi_{\tau^{(n)}_1}=\begin{scriptsize}
\begin{pmatrix}
x_{n-1} \\ -x_{n}
\end{pmatrix}
\end{scriptsize}$ respectively of $\I_{\tau^{(1)}_1},\ldots, \I_{\tau^{(n)}_1}$ completed with zeros entries.

This work takes it roots in trying to generalize this previous construction and, given two integers $m,m'\geqslant 1$, to understand the projective degrees of such an outputted glued determinantal map $\gm{g}{g'}:\p{m+m'}_\kk\dashrightarrow\p{m+m'}_\kk$ via the properties of the inputted determinantal maps $g:\p{m}_\kk\dashrightarrow\p{m}_\kk$ and $g':\p{m'}_\kk\dashrightarrow\p{m'}_\kk$. Recall that given any map $f=(f_0:\ldots:f_n):\p{n}_\kk\dashrightarrow\p{n}_\kk$ and assuming that the base field $\kk$ is algebraically closed, the projective degrees $\d{f}{0},\d{f}{1}\ldots, \d{f}{n}$ of $f$ are the quantities such that for any $k\in\lbrace 0,\ldots,n\rbrace$, $\d{f}{k}$ is the cardinal $\#\big( H_\mbf{x}^{k}\cap f^{-1}(H_\mbf{y}^{n-k})\big)$ of the intersection between a general linear space $H_{\mbf{x}}^{k}$ of codimension $k$ in the source space of $f$ with the preimage of a general linear space $H_\mbf{y}^{n-k}$ of codimension $n-k$ in the target space of $f$, see \cite[Ex. 19.4]{harris1992algebraic} or below for more developments about the projective degrees of a map. Given this latter definition, let us state a  first result describing, in some situations, the projective degree of $\gm{g}{g'}$.
\begin{prpIntro}\label{prpKunneth}
Let $m,m'\geqslant 1$, $\RR_m=\kk[x_0,\ldots,x_m]$, $\RR_{m'}=\kk[x_m,\ldots,x_{m+m'}]$ and $\RR=\kk[x_0,\ldots,x_{m+m'}]$ and let $g:\p{m}_\kk\dashrightarrow \p{m}_\kk$ (resp. $g':\p{m'}_\kk\dashrightarrow\p{m'}_\kk$) be a determinantal map such that $\I_{g}$ (resp. $\I_{g'}$) is the $m$-minors ideal of a matrix $\Phi_{g}\in\RR_m^{(m+1)\times m}$ (resp. the $m'$-minors ideal of $\Phi_{g'}\in\RR_{m'}^{(m'+1)\times m'}$).

Put
\[
\begin{tikzpicture}[
style1/.style={
  matrix of math nodes,
  every node/.append style={text width=#1,align=center,minimum height=5ex},
  nodes in empty cells,
  left delimiter=(,
  right delimiter=),
  }
  ]
\matrix[style1=0.85cm] (1mat)
{
  &   \\
  &  \\
  &  \\
};
\draw[dashed]
  (1mat-2-2.west) -- (1mat-2-2.east);
\draw[dashed]
  (1mat-2-1.south west) -- (1mat-2-1.south east);
\draw[dashed]
  (1mat-1-1.north east) -- (1mat-3-1.south east);

\node
  at (1mat-1-1.south) {$\Phi_g$};
\node
  at (1mat-1-2) {$0_{m\times m'}$};
\node 
  at (1mat-3-1) {$0_{m'\times m}$};
\node
  at (1mat-3-2) {$\Phi_{g'}$};

\begin{scriptsize}
\draw[decoration={brace,raise=12pt},decorate]
  (1mat-1-2.north east) -- 
  node[right=15pt] {$m$} 
  (1mat-2-2.east);
\draw[decoration={brace,amplitude=5pt,raise=12pt},decorate]
  (1mat-2-2.east) -- 
  node[right=15pt] {$m'+1$} 
  (1mat-3-2.south east);
\draw[decoration={brace,raise=7pt},decorate]
  (1mat-1-1.north west) -- 
  node[above=8pt] {$m$} 
  (1mat-1-1.north east);
\draw[decoration={brace,raise=7pt},decorate]
  (1mat-1-2.north west) -- 
  node[above=8pt] {$m'$} 
  (1mat-1-2.north east);
  
\draw[decoration={mirror, brace,amplitude=5pt,raise=11pt},decorate]
  (1mat-1-1.north west) -- 
  node[left=14pt] {$m+1$} 
  (1mat-2-1.south west);
\draw[decoration={mirror, brace,raise=12pt},decorate]
  (1mat-3-1.north west) -- 
  node[left=14pt] {$m'$} 
  (1mat-3-1.south west);
 
\end{scriptsize}
\node [left=70pt]{$\Phi_{\gm{g}{g'}}=$};
\node[right=60pt] {$\in\RR^{(m+m'+1)\times (m+m') }$};
\end{tikzpicture}
\]
and assume that the ideal of $(m+m')$-minors of $\Phi_{\gm{g}{g'}}$ is the base ideal of a map $\gm{g}{g'}:\p{m+m'}_\kk\dashrightarrow \p{m+m'}_\kk$.

Suppose that the tensor product $\mathbb{F}_{g}\otimes \mathbb{F}_{g'}$ of a free resolution $\mathbb{F}_{g}$ of the graph of $g$ and a free resolution $\mathbb{F}_{g'}$ of the graph of $g'$ provides a free resolution of the graph of $\gm{g}{g'}$. Then, given any $k\in \lbrace 0,\ldots,m+m'\rbrace$
\begin{equation}\label{eqKunnethIntro}\tag{1}
\d{\gm{g}{g'}}{k}=\underset{p=0}{\overset{k}{\sum}}\d{g}{p}\d{g'}{k-p}
\end{equation} 
with the convention that $\d{g}{p}=0$ (resp. $\d{g'}{p}=0$) if $p>m$ (resp. $p>m'$).
\end{prpIntro}

As we will explain, \Cref{prpKunneth} applies in particular in the previous example of the standard Cremona map.

Actually, other definitions of a glued determinantal map $\gm{g}{g'}$ can be given depending on the convention adopted to concatenate the matrices $\Phi_g$ and $\Phi_{g'}$. For instance, in the case where each entry of any given column of $\Phi_g$ and $\Phi_{g'}$ is a general combination of given polynomials, a pseudo concatenate matrix $\Phi_{\gm{g}{g'}}$ and an associated glued map $\gm{g}{g'}$ can still be defined. In this direction:

\begin{prpIntro2}
Let $m,d\geqslant 1$ and let $\Phi_{\gm{g}{g'}}=(\phi_{ij})_{\genfrac{}{}{0pt}{}{0\leqslant i \leqslant m+2}{1\leqslant j \leqslant m+2}}$ be such that:
\begin{itemize}
\item for all $l\in\lbrace 1,\ldots, m\rbrace$, all the entries $\phi_{il}$ of the $l$-th column of $\Phi_{\gm{g}{g'}}$ are general linear combinations of $x_0,\ldots,x_{m}$,
\item all the entries $\phi_{i,m+1}$ of the $m+1$-th column of $\Phi_{\gm{g}{g'}}$ are general linear combinations of $x_{m+1}$ and $x_{m+2}$,
\item  all the entries $\phi_{i,m+2}$ of the $m+2$-th column of $\Phi_{\gm{g}{g'}}$ are general linear combinations of the generators of the ideal \[(x_m,x_{m+1},x_{m+2})\cdot(x_{m+1},x_{m+2})^{d-1}=(x_mx_{m+1}^{d-1},x_mx_{m+2}^{d-1})+(x_{m+1},x_{m+2})^{d}.\]
\end{itemize} 

Then the map $\gm{g}{g'}:\p{m+2}_\kk\dasharrow\p{m+2}_\kk$ whose base ideal $\I_{\gm{g}{g'}}$ is the $(m+2)$-minors ideal of $\Phi_{\gm{g}{g'}}$ is a determinantal map and moreover:
\[\forall k\in \lbrace 0,\ldots,m+2\rbrace,\,\d{\gm{g}{g'}}{k}=\binom{m}{m-k}+(d+1)\binom{m}{m-k+1}+\binom{m}{m-k+2}\] with the convention that $\binom{j}{i}=0$ if $i<0$ or $i>j$. In particular, the projective degrees' vector $\d{\gm{g}{g'}}{}=\big(\d{\gm{g}{g'}}{0},\ldots,\d{\gm{g}{g'}}{m+2}\big)$ of $\gm{g}{g'}$ is palindromic.
\end{prpIntro2}

Our goal behind this glued construction is to approach combinatorially the projective degrees of the \emph{almost linear determinantal maps}, i.e.\ the determinantal map $f$ defined by \emph{almost linear matrix} $\Phi_f$ for which all but one column are filled with linear polynomials (see \cite{KusPolUl2011RatNormScrolls} for development about maps defined by almost linear matrices), in order to eventually extend in greater dimension previous works such as in \cite{DesHan2017QQ}. 

The proof of \Cref{prpMMGluedMap} relies on the following observations. When considering a map $f:\p{n}_\kk\dashrightarrow \p{n}_\kk$ of base ideal $ \I_f$ with presenting matrix $\Phi_f$, the scheme $\p{}(\I_f)=\Proj\big(\Sym(\I_f)\big)$ provides an approximation of the projective degrees of $f$. This is now a classical approach, see \cite{RussoSimis2001OnBirMap}, \cite{BuChJo2009TorSymAlg}, \cite{BuCiDAnd2018MultiGrad} and \cite{CidRuiz2021MixedMultProjDeg} for developments and applications in broader contexts. If $f$ is in addition \emph{Koszul-determinantal} i.e.\ $\p{}(\I_f)$ is a complete intersection in its embedding in $\p{n}_\kk\times\p{n}_\kk$ defined by $\Phi_f$ (a condition on $f$ stronger than being determinantal and that we will explicit below), one can use the resolution of zero-dimensional complete intersection defined by sparse polynomials to compute, in some situations, the projective degrees of $f$. Indeed, Bernstein's theorem states that, under conditions, the mixed volume associated to a polynomial system $(E)$ compute the solutions of $(E)$ with non zero coordinates, see \cite[7. Th. 5.4]{coXLittleOShea2005UsingAlgGeo} or \Cref{thmBernstein} below. \Cref{prpMMGluedMap} then follows from convex geometry, and in particular the projection formula in \cite[Lemma 6]{SteffensTheobald2010MixedVolumesTechniques} decomposing some mixed volumes as the product of mixed volumes in smaller dimension.

\subsection*{Contents of the paper} In \Cref{SecKunneth}, we precise the definition of the projective degrees of a map $f$, emphasizing how they are related to (bi-graded) free resolutions of the ideal of the graph $\Gamma_f$ of $f$. It gives the background of our proof of \Cref{prpKunneth}. In \Cref{SecMixVol}, we precise the approximations of the projective degrees of a Koszul-determinantal map given by the mixed volumes associated to the presentation matrix its the base ideal. We then give another definition of a glued determinantal map $\gm{g}{g'}:\p{m+m'}_\kk\dashrightarrow\p{m+m'}_\kk$ via general properties of two initial determinantal maps $g:\p{m}_\kk\dashrightarrow\p{m}_\kk$ and $g':\p{m'}_\kk\dashrightarrow\p{m'}_\kk$. Under assumptions on $g$ and $g'$, the projective degrees of such a glued map $\gm{g}{g'}$ are then described by a Kunneth-like formula in the projective degrees of $g$ and $g'$, see \Cref{prpMMGluedMap} and its proof.

The explicit computations given in this paper were made using basic functions of the software systems \textsc{Polymake} and  \textsc{Macaulay2} and the \textsc{Cremona} package  \cite{stagliano2017Mac2Pack} associated. The corresponding codes are available on request.

\subsection*{Acknowledgements} I thank Adrien Dubouloz and Daniele Faenzi for their comments about a first version of this manuscript. I also warmly thank Alessandro de Stefani for its help on some arguments of the first section.

\begin{ntt}\label{ntt}
We let $\kk$ be any field in \Cref{SecKunneth} and $\kk=\mathbb{C}$ in \Cref{SecMixVol}.

When recalling generalities about rational maps, we consider a positive integer $n\in\mathbb{N}^*$, $\msf{R}=\kk[\msf{x}_0,\ldots,\msf{x}_0]$, $\p{n}_\kk=\Proj(\msf{R})$ and a rational map $f:\p{n}_\kk\dashrightarrow \p{n}_\kk$.

When considering glued maps, we let $m,m'\geqslant 1$, $\RR_m=\kk[x_0,\ldots,x_{m}]$, $\RR_{m'}=\kk[x_m,\ldots,x_{m+m'}]$, $\RR=\kk[x_0,\ldots,x_{m+m'}]$ and we always consider that $\RR_m$ and $\RR_{m'}$ are embedded in $\RR$. Let also $g:\p{m}_\kk\dashrightarrow \p{m}_\kk$ (resp. $g':\p{m'}_\kk\dashrightarrow\p{m'}_\kk$),  where $\p{m}_\kk=\Proj(\RR_m)$ (resp. $\p{m'}_\kk=\Proj(\RR_{m'})$). The associated glued map is denoted $\gm{g}{g'}:\p{m+m'}_\kk\dashrightarrow\p{m+m'}_\kk$ where $\p{m+m'}_\kk=\Proj(\RR_{m+m'})$.
\end{ntt}

\section{Free resolutions of the graph of glued determinantal maps}\label{SecKunneth}
Our motivation in describing a free resolution of the ideal of the graph of a map $f$ (free resolution of the graph for short) lies on the fact it gives an insight on the projective degrees of $f$ and let us briefly outline why.

As explained in \cite[Chapter 8]{MilSturm2005CombCA}, the theory of multigraded Hilbert series is the fine context for defining and studying the multidegree of subschemes in multiprojective spaces so we adapt here previous expositions as in \cite{MilSturm2005CombCA} or the more recent one in \cite{CidRuiz2021MixedMultProjDeg} to the following bi-graded context. Let $n\geqslant 1$ be an integer, $\kk$ be any field (if no additional specific assumptions), $\msf{R}=\kk[\mathsf{x}_0,\ldots,\mathsf{x}_n]$ and $\msf{S}=\mathsf{R}[\mathsf{y}_0,\ldots,\mathsf{y}_n]$ be the $\mathbb{Z}^2$-graded polynomial ring with the standard graduation $\deg(\mathsf{x}_i)=(1,0)$ and $\deg(\mathsf{y}_i)=(0,1)$ for all $i\in\lbrace 0,\ldots,n\rbrace$. Given a bi-graded algebra $\BB$, we let $\BiProj(\BB)$ be the set of bi-graded prime ideals of $\BB$. Thus, a subscheme $\X$ of $\BiProj(\msf{S})\simeq \p{n}_\kk\times\p{n}_\kk$ is defined by a bi-graded ideal $\A\subset\msf{S} $ and one associates to $\X=\BiProj(\msf{S}/\A)$ its bi-variate Hilbert series \[\Hilb_{\msf{S}/\A}(T_0,T_1)=\sum_{(n_0,n_1)\in\mathbb{Z}^2}(\lgth(\msf{S}/\A)_{(n_0,n_1)}T_0^{n_0}T_1^{n_1}.\] Then, writing $\Hilb_{\msf{S}/\A}(T_0,T_1)=\frac{\Num_{\msf{S}/\A}(T_0,T_1)}{(1-T_0)^l(1-T_1)^l}$ and assuming that $\X=\BiProj(\msf{S}/\A)$ has codimension $c$, the coefficients of the homogeneous component of $\Num_{\msf{S}/\A}(1-T_0,1-T_1)$ of total degree $c$ define the multidegree of $\X$.

\begin{dfn}[multidegree of a subscheme in $\p{n}_\kk\times\p{n}_\kk$]
Let $\A$ be a bi-graded ideal of $\msf{S}$ and assume that $\X=\BiProj(\msf{S}/\A)$ has codimension $c$ in $\BiProj(\msf{S})\simeq \p{n}_\kk\times\p{n}_\kk$.

Given $k \in \lbrace 0,\ldots,c \rbrace$, define $\e{\X}{c-k,k}$ as the coefficient of the monomial of bi-degree $(c-k,k)$ of $\Num_{\msf{S}/\A}(1-T_0,1-T_1)$.

The \emph{multidegree} $\e{\X}{}$ of $\X$ is the $c+1$-uple \[\e{\X}{}=\big(\e{\X}{c,0},\ldots,\e{\X}{0,c}\big).\]
\end{dfn}

\begin{rmk}\label{rmkCornerstone}
As explained for instance in \cite[Proposition 7.16]{harris1992algebraic} or \cite[Theorem 4.7]{CidRuiz2021MixedMultProjDeg}, the multidegree $\e{\X}{}=\big(\e{\X}{c,0},\ldots,\e{\X}{0,c}\big)$ of a $c$-dimensional subscheme $\X\subset \p{n}_\kk\times\p{n}_\kk$ has the following geometric interpretation. Assuming the base field $\kk$ is algebraically closed and letting $k \in \lbrace 0,\ldots, c\rbrace$, we have:
\[\e{\X}{c-k,k}=\lgth(H_\mbf{x}^{k}\cap \X\cap H_\mbf{y}^{c-k})\]
where $H_\mbf{x}^{k}=\V(l_{1,0},\ldots,l_{k,0})$ is the zero locus of $k$ general linear forms $l_{1,0},\ldots,l_{k,0}$ in the $\mathsf{x}$-variables (that is $\deg(l_{i,0})=(1,0)$ for all $i\in\lbrace 0,\ldots,k\rbrace$) and $H_\mbf{y}^{k}=\V(l_{0,1},\ldots,l_{0,c-k})$ is the zero locus of $c-k$ general linear forms $l_{0,1},\ldots,l_{0,c-k}$ in the $\mathsf{y}$-variables.
\end{rmk}

\begin{dfn}[projective degrees of a map]
Let $f:\p{n}_\kk\dasharrow\p{n}_\kk$, regular on a dense open subset $U\subset\p{n}_\kk$, and let $\G_f=\overline{\lbrace \big(\mathsf{x},f(\mathsf{x})\big), \msf{x}\in U\rbrace}\subset\p{n}_\kk\times\p{n}_\kk$ be the graph of $f$. It is an $n$-fold so, for $k\in\lbrace 0,\ldots,n\rbrace$, the \emph{$k$-th projective degree of $f$}, written $\d{f}{k}$, is defined by \[\d{f}{k}:=\e{\G_f}{n-k,k}.\] In the following, our convention is also that $\d{f}{k}=0$ if $k>n$.
\end{dfn}
By \Cref{rmkCornerstone}, note that $\d{f}{0}=1$ if and only if $f$ is an isomorphism between two dense open subsets of $\p{n}_\kk$ in which case $f$ is said to be a \emph{Cremona map}.

Following for instance \cite[III.3.6]{eisenbud2000geo}, a bigraded free resolution of $S/\I_{\Gamma_f}$ provides the numerator $\Num_{\msf{S}/\I_f}(T_0,T_1)$ of the Hilbert series of the coordinate ring $\msf{S}/\I_{\Gamma_f}$ of $\Gamma_f$ and thus provides the projective degrees of $f$.

Now let us clarify in which embedding we consider the graphs $\Gamma_g$ and $\Gamma_{g'}$ of, respectively, a map $g:\p{m}_\kk\dashrightarrow\p{m}_\kk$ and $g':\p{m'}_\kk\dashrightarrow\p{m'}_\kk$ see \Cref{ntt}. Denote by $\SS_m=\RR_m[y_0,\ldots,y_m]$, $\SS_{m'}=\RR_{m'}[y_m,\ldots,y_{m+m'}]$, $\SS=\RR[y_0,\ldots,y_{m+m'}]$ and $\I_{\Gamma_g}\subset \SS$ (resp. $\I_{\Gamma_{g'}}\subset \SS$) be the ideal of the graph $\Gamma_g\subset \p{m+m'}_\kk\times  \p{m+m'}_\kk$ of $g$ (resp. $\Gamma_{g'}\subset \p{m+m'}_\kk\times  \p{m+m'}_\kk$ of $g'$) where $\Gamma_g\subset \p{m+m'}_\kk\times  \p{m+m'}_\kk=\BiProj(\SS)$. Geometrically, $\Gamma_g$ (resp. $\Gamma_{g'}$) is thus the cone of vertex $\Gamma_g\cap \p{m}_\kk\times  \p{m}_\kk$ (resp. $\Gamma_{g'}\cap \p{m'}_\kk\times  \p{m'}_\kk$) where $\p{m}_\kk\times  \p{m}_\kk=\BiProj(\SS_m)$ (resp. $\p{m'}_\kk\times  \p{m'}_\kk=\BiProj(\SS_{m'})$. \Cref{prpKunneth} follows from the following results about $\Gamma_g\cap \Gamma_{g'}$.

\begin{lmm}\label{lmmKunneth} Assume that the tensor product $\mathbb{F}_g\otimes\mathbb{F}_{g'}$, of a bigraded free resolution $\mathbb{F}_g$ of $\Gamma_g$ with a bigraded free resolution $\mathbb{F}_{g'}$ of $\Gamma_{g'}$, is a free resolution of $\Gamma_g\cap \Gamma_{g'}$. Then for any $k\in\lbrace 0,\ldots,m+m'\rbrace$:
\begin{equation*}\d{\gm{g}{g'}}{k}=\underset{p=0}{\overset{k}{\sum}}\d{g}{p}\d{g'}{k-p}.
\end{equation*}
\end{lmm}

\begin{proof}
By assumptions, the numerator $\Num_{\SS/\I_{\Gamma_g\cap \Gamma_{g'}}}(T_0,T_1)$ of the Hilbert series of $\SS/\I_{\Gamma_g\cap \Gamma_{g'}}$ is the product $\Num_{g}(T_0,T_1)\Num_{g'}(T_0,T_1)$ of the numerators of the Hilbert series of respectively $\SS/\I_{g}$ and $\SS/\I_{g'}$.

Focusing on the homogeneous component of total degree $m+m'$ of $\Num_{\gm{g}{g'}}(1-T_0,1-T_1)$, we thus have:
\begin{align*}
\big(Num_{\SS/\I_{\Gamma_g\cap \Gamma_{g'}}}(1-T_0,&1-T_1)\big)_{m+m'}\\
&=\big(\Num_{g}(1-T_0,1-T_1)\Num_{g'}(1-T_0,1-T_1)\big)_{m+m'}\\
&=\big(\Num_{g}(1-T_0,1-T_1)\big)_m\big(\Num_{g'}(1-T_0,1-T_1)\big)_{m'}
\end{align*}
the last equality holds because $\big(\Num_{g}(1-T_0,1-T_1)\big)_m$ (resp. $\big(\Num_{g'}(1-T_0,1-T_1)\big)_{m'}$) is also the homogeneous component of smallest total degree of $\Num_{g}(1-T_0,1-T_1)$ (resp. $\Num_{g}'(1-T_0,1-T_1)$) as $\Gamma_g$, being the cone over the graph of $g$, is irreducible of codimension $m$ (resp. $\Gamma_{g'}$ is irreducible of codimension $m'$).
\end{proof}

Given these preliminaries and before proving \Cref{prpKunneth}, let us briefly underline why, when building a glued map $\gm{g}{g'}$, we restrict to inputted determinantal maps $g$ and $g'$.
\begin{dfn}[Hilbert-Burch matrix of a determinantal map]\label{dfnDetMap} Given that the base ideal $\I_f=(f_0:\ldots:f_n)\subset\msf{R}$ of $f$ verifies $\codim\big(\V(\I_f)\big)\geqslant 2$ and that $\I_f$ is the $n$-minors ideal of a $(n+1)\times n$-matrix, Hilbert-Burch theorem \cite[Theorem 20.15]{eisenbud1995algebra} states that $\I_f$ is the $n$-minors ideal of its presentation matrix $\Phi_f$, i.e.\ the presentation of $\I_f$ reads:
\begin{equation*}
\begin{tikzcd}[row sep=0.8em,column sep=1em,minimum width=2em]
0\ar[r]& \msf{R}^n\ar[r, "{\Phi_f}"]& \msf{R}^{n+1} \ar[rrr,"{(f_0 \; \ldots \; f_n)}"] & & & \I_f \ar{r}& 0.
\end{tikzcd}
\end{equation*} 
It motives why the matrix $\Phi_f$ is called the \emph{Hilbert-Burch matrix} of $f$ in the following, see for instance \cite{KusPolUl2011RatNormScrolls}, \cite{KusPolUl2013BiGrStructSymAlg} and \cite{KimMuk2020EqDefCertGraphs} for recent works about maps defined by Hilbert-Burch matrices.
\end{dfn}

When gluing two determinantal maps $g$ and $g'$ as in \Cref{prpKunneth}, the $(m+m')$-minors ideal of the concatenated matrix $\Phi_{\gm{g}{g'}}$ always defines a map $\p{m+m'}_\kk\dashrightarrow \p{m+m'}_\kk$ and, if $\codim \V\big(\I_{m+m'}(\Phi_{\gm{g}{g'}})\big)=2$, the base ideal $\I_{\gm{g}{g'}}$ of $\gm{g}{g'}$ is equal to $\I_{m+m'}(\Phi_{\gm{g}{g'}})$. Let us mention here that a glued map $\gm{g}{g'}$ can also be defined when $g$ and $g'$ are not necessarily determinantal however in this case $\gm{g}{g'}$ is not necessarily a map from $\p{m+m'}_\kk$ to $\p{m+m'}_\kk$.

Now recall that, algebraically, the graph $\G_f$ of a map $f=(f_0:\ldots:f_n):\p{n}_\kk\dasharrow\p{n}_\kk$ is the $\Proj$ of the Rees algebra $\R(\I_f)=\underset{k\geqslant 0}{\oplus}\I_f^kt^k$ of the base ideal $\I_f$ of $f$ and the embedding of $\G_f$ in $\p{n}_\kk\dasharrow\p{n}_\kk$ is defined by the surjection $\msf{S}\twoheadrightarrow \R(\I_f)$ sending the variables $\msf{y}_i$ to $f_it$ ($\msf{S}=\msf{R}[\msf{y}_0,\ldots,\msf{y}_n]$). In practice, the kernel $J$ of the latter surjection may be difficult to compute and an approximation of $\J$ is more accessible by considering the symmetric algebra of $\I_f$, this is a classical approach that we now summarize briefly in our context of rational maps, see \cite{Vasconcelos2005Int} for an introduction in a broader context and pointers to references about this procedure. The presentation matrix $\Phi_f$ of $\I_f$, by definition verifying the following short exact sequence:
\begin{equation*}
\begin{tikzcd}[row sep=0.8em,column sep=1em,minimum width=2em]
\msf{R}^{n'}\ar[r, "{\Phi_f}"]& \msf{R}^{n+1} \ar[rrr,"{(f_0 \; \ldots \; f_n)}"] & & & \I_f \ar{r}& 0,
\end{tikzcd}
\end{equation*} defines an embedding of $\p{}(\I_f)=\Proj\big(\Sym(\I_f)\big)$ in $\p{n}_\kk\times\p{n}_\kk$ whose ideal $\J_1$ is generated by the $n'$ entries of the line matrix $(\msf{y}_0 \; \ldots \; \msf{y}_n)\Phi_f$ (here we consider that $\Phi_f$ is a matrix both in $\msf{R}$ and $\msf{S}$). Since the natural surjection $\Sym(\I_f)\twoheadrightarrow\R(\I_f)$ factorizing the maps $\I_f^{\otimes k}\twoheadrightarrow\I_f^k$ defines an embedding of $\G_f$ in $\p{}(\I_f)$, one has that $\J_1\subset \J$, i.e.\ $\J_1$ provides some equations of $\G_f$. The ideal $\I_f$ is said to be of \emph{linear type} when $\J_1= \J$.

When $f$ is determinantal of Hilbert-Burch matrix $\Phi_f\in\msf{R}^{(n+1)\times n}$, the kernel of $\Sym(\I_f)\twoheadrightarrow\R(\I_f)$ is the $\msf{R}$-torsion of $\Sym(\I_f)$ \cite{Micali1964AlgUniv} and are described by the Fitting ideals of $\I_f$, i.e.\ by the ideals $\I_k(\Phi_f)$ for $k\in\lbrace 1,\ldots,n-1\rbrace$ of $k$-minors ideals of the presentation matrix $\Phi_f$ of $\I_f$, see \cite[Prop. 1.1 and below]{Vasconcelos2005Int} and \cite{BuChJo2009TorSymAlg}. Hence, the irreducible components of  $\p{}(\I_f)$ are the graph $\Gamma_f$ of $f$ and eventual additional pieces lying above closed strict subschemes of the source space $\p{n}_\kk$ of $f$.

With all these facts, let us now show \Cref{prpKunneth}.

\begin{proof}[Proof of \Cref{prpKunneth}]
Assume that $\codim \V\big(\I_{m+m'}(\Phi_{\gm{g}{g'}})\big)=2$ so the base ideal $\I_{\gm{g}{g'}}$ of the glued map $\gm{g}{g'}:\p{m+m'}_\kk\dashrightarrow \p{m+m'}_\kk$ is the $(m+m')$-minors ideal of the matrix $\Phi_{\gm{g}{g'}}$. It is clear that $\Gamma_{\gm{g}{g'}}$ is included in $\Gamma_g\cap \Gamma_{g'}$. Indeed, by definition, $\Gamma_{\gm{g}{g'}}\subset \p{}(\I_{\gm{g}{g'}})=\p{}(\I_{g})\cap \p{}(\I_{g'})$ so $\Gamma_{\gm{g}{g'}}\subset \p{}(\I_{g})$ (resp. $\Gamma_{\gm{g}{g'}}\subset \p{}(\I_{g'})$) and since $\Gamma_{\gm{g}{g'}}$ cannot be included in a component of $\p{}(\I_{g})$ (resp. $\p{}(\I_{g'})$) lying above a closed strict subscheme of $\p{m+m'}_\kk$, it is necessarily included in $\Gamma_{g}$ (resp. $\Gamma_{g'}$).

Now, since, by assumption, $\mathbb{F}_g\otimes\mathbb{F}_{g'}$  is a resolution of $\Gamma_{\gm{g}{g'}}$ and $\Gamma_{\gm{g}{g'}}\subset \Gamma_g\cap \Gamma_{g'}$, one has $\Gamma_{\gm{g}{g'}}= \Gamma_g\cap \Gamma_{g'}$.

The conclusion of \Cref{prpKunneth} follows then from \Cref{lmmKunneth}.
\end{proof}

Note that, as set, $\Gamma_{\gm{g}{g'}}$ and $\Gamma_g\cap \Gamma_{g'}$ are not always equal as illustrated by the following example.

\begin{ex}
Let $\Phi_g=\begin{pmatrix}
x_2 & 0 \\
x_1 & x_0x_2 \\
 0  & x_1^2
\end{pmatrix}$ and $\Phi_{g'}=\begin{pmatrix}
x_3 & 0 \\
x_4  & x_2x_3 \\
 0  & x_4^2
\end{pmatrix}$ and $\Phi_{\gm{g}{g'}}$ be the concatenate matrix of $\Phi_g$ and $\Phi_{g'}$ as in \Cref{prpKunneth}. Then, as it can be checked by direct computation on a computer algebra system such as \textsc{Macaulay2}, one has that $\codim\V\big(\Phi_{\gm{g}{g'}})\big)\geqslant 2$ and that $\d{\gm{g}{g'}}{}=(1,6,10,6,1)$, $\d{g}{}=\d{g'}{}=(1,3,1)$ and \[\d{\gm{g}{g'}}{2}=10<11=\d{g}{0}\d{g'}{2}+\d{g}{1}\d{g'}{1}+\d{g}{2}\d{g'}{0}.\]
In this example $\mathbb{F}_g\otimes\mathbb{F}_{g'}$ provides a bigraded free resolution of $\Gamma_g\cap \Gamma_{g'}$ and $\Gamma_{\gm{g}{g'}}\subsetneq \Gamma_g\cap \Gamma_{g'}$.
\end{ex}

Let us also mention that there exist situations where the ideal of $\Gamma_g\cap \Gamma_{g'}$ decomposes as the intersection of the ideal of $\Gamma_{\gm{g}{g'}}$ and the ideal of embedded components, so that $\Gamma_{\gm{g}{g'}}$ and $\Gamma_g\cap \Gamma_{g'}$ are not scheme-theoretically equal but have the same projective degrees.

Even if it might seem a bit artificial (since the results about the standard Cremona maps are well known, in particular its projective degrees, and computable by other means, see \cite{PanGonzalez2005CarClassesDetCremTrans}), we illustrate an application of \Cref{prpKunneth} in the example of the standard Cremona maps $\tau_{m+m'}=\gm{\tau_m}{\tau_{m'}}$.
\begin{prp}\label{prpAppliKunnethStdCrem}
Following \Cref{ntt}, let \begin{align*}
&\tau_m=(x_1\cdots x_m:\ldots:x_0\cdots x_{m-1}):\p{m}_\kk\dashrightarrow \p{m}_\kk\\
&\tau_{m'}=(x_{m+1}\cdots x_{m+m'}:\ldots:x_m\cdots x_{m+m'-1}):\p{m'}_\kk\dashrightarrow \p{m'}_\kk
\end{align*} with associated Hilbert-Burch matrices $\Phi_{\tau_m}\in\RR_m^{(m+1)\times m}$ and $\Phi_{\tau_{m'}}\in\RR_{m'}^{(m'+1)\times m'}$. Let also $\Phi_{\gm{\tau_m}{\tau_{m'}}}\in \RR^{(m+m'+1)\times (m+m')}$ be as in \Cref{prpKunneth}. Then:
\begin{enumerate}
\item\label{item1CrmStd} $\codim\V(\I_{m+m'}\big(\Phi_{\gm{\tau_m}{\tau_{m'}}})\big)=2$,
\item\label{item2CrmStd}  the ideal of $\p{}(\I_{\gm{\tau_m}{\tau_{m'}}})$ is minimally resolved by $\mathbb{F}_{\tau_m}\otimes \mathbb{F}_{\tau_{m'}}$ which is the Koszul complex on the entries of the line matrix $\begin{pmatrix}
y_0 & \ldots & y_{m+m'}
\end{pmatrix}\Phi_{\gm{\tau_m}{\tau_{m'}}}$.
\item\label{item3CrmStd} $\p{}(\I_{\gm{\tau_m}{\tau_{m'}}})=\Gamma_{\gm{\tau_m}{\tau_{m'}}}$ and, consequently, for any $k\in\lbrace 0,\ldots,m+m'\rbrace $
\begin{equation}\label{eqProjDegStdCrem}
\d{\tau_{m+m'}}{k}=\underset{p=0}{\overset{k}{\sum}}\d{\tau_m}{p}\d{\tau_{m'}}{k-p}=\binom{m+m'}{k}.
\end{equation}
\end{enumerate}
\end{prp}

\begin{proof}
Let $m,m'\geqslant 1$ and \begin{equation}\label{eqMatGlue}
\Phi_{\tau_{m+m'}}=\begin{scriptsize}\begin{pmatrix}
   x_0  &    0   & \ldots & \ldots & 0      \\
   -x_1 &   x_1  & \ddots &        & \vdots \\
    0   &  -x_2  & \ddots & \ddots & \vdots \\
 \vdots & \ddots & \ddots & \ddots &    0   \\
 \vdots &        & \ddots & \ddots & x_{m+m'-1}\\
    0   & \ldots & \ldots &    0   & -x_{m+m'}.
\end{pmatrix}\end{scriptsize}
\end{equation}

Since \[\I_{m+m'}(\Phi_{\gm{\tau_{m}}{\tau_{m'}}})=(x_1\cdots x_{m+m'},\ldots,x_0\cdots x_{m+m'-1}),\] one has that $\codim \V\big(\I_{m+m'}(\Phi_{\gm{\tau_{m}}{\tau_{m'}}})\big)=2$ (the generators of $\I_{m+m'}(\Phi_{\gm{\tau_{m}}{\tau_{m'}}})$ do not share a common factor) and thus \Cref{item1CrmStd} is verified.

The proof of \Cref{item2CrmStd} and \Cref{item3CrmStd} relies on an induction on $m$ and $m'$. The inductive step consists in first showing that the tensor product $\mathbb{F}_{\tau_m}\otimes \mathbb{F}_{\tau_{m'}}$ of a free resolution $\mathbb{F}_{\tau_m}$ of $\p{}(\I_{\tau_{m}})$ and a free resolution $\mathbb{F}_{\tau_{m'}}$ of $\p{}(\I_{\tau_{m'}})$ provides a free resolution of $\p{}(\I_{\tau_{m}})\cap \p{}(\I_{\tau_{m'}})=\p{}(\I_{\gm{\tau_{m}}{\tau_{m'}}})$. The second step consists in showing that $\p{}(\I_{\gm{\tau_{m}}{\tau_{m'}}})$ is irreducible or, in other words, that $\I_{\p{}(\I_{\tau_{m}})}+\I_{\p{}(\I_{\tau_{m'}})}$ is prime. This latter property insures that \[\Gamma_{\tau_m}\cap \Gamma_{\tau_{m'}}=\p{}(\I_{\tau_{m}})\cap\p{}(\I_{\tau_{m'}})=\p{}(\I_{\gm{\tau_{m}}{\tau_{m'}}})=\Gamma_{\gm{\tau_{m}}{\tau_{m'}}}\] and \Cref{eqProjDegStdCrem} follows then from a direct application of \Cref{prpKunneth}.

\begin{itemize}
\item Initial case: \Cref{item2CrmStd} and \Cref{item3CrmStd} are verified in the case $m=1$, this follows from the fact that $\p{}(\I_{\tau_1})=\V(x_0y_0-x_1y_1)$ whose minimal free resolution is the Koszul complex on the single irreducible polynomial $x_0y_0-x_1y_1\in \kk[x_0,x_1,y_0,y_1]$.
\item Inductive step: let $m,m'\geqslant 1$ and assume that \Cref{item2CrmStd} and \Cref{item3CrmStd} hold for $\tau_m$ and $\tau_{m'}$. In particular the ideal $\mathcal{I}_{m}$ of $\p{}(\I_{\tau_{m}})=\Gamma_{\tau_{m}}$ (resp. the ideal $\mathcal{I}_{m'}$ of $\p{}(\I_{\tau_{m'}})=\Gamma_{\tau_{m'}}$) is minimally resolved by the Koszul complex $\mathbb{F}_{\tau_{m}}$ on the entries of the line matrix $\begin{pmatrix}
y_{0} & \ldots & y_{m}
\end{pmatrix}\Phi_{\tau_{m}}$ (resp. by the Koszul complex $\mathbb{F}_{\tau_{m'}}$ on the entries of the line matrix $\begin{pmatrix}
y_{m} & \ldots & y_{m'}
\end{pmatrix}\Phi_{\tau_{m'}}$). We show \Cref{item2CrmStd} and \Cref{item3CrmStd} via classical methods involving Gr\"obner bases and we refer to \cite{AL1994GB} for the associated definitions associated to this argument. Remark that $\mathbb{F}_{\tau_{m}}\otimes \mathbb{F}_{\tau_{m+m'}}$ is the Koszul complex on the entries of $\begin{pmatrix}
y_{0} & \ldots & y_{m+m'}
\end{pmatrix}\Phi_{\gm{\tau_{m}}{\tau_{m'}}}$ and one has that $\mathbb{F}_{\tau_{m}}\otimes \mathbb{F}_{\tau_{m'}}$ resolves the ideal of $\p{}(\I_{\gm{\tau_{m}}{\tau_{m'}}})$ if and only if
\[\forall i\geqslant 1,\Tor^i\big(\SS/\mathcal{I}_m,\SS/\mc{I}_{m'}\big)=0,\] $\SS$ being the polynomial ring $\kk[x_0,\ldots,x_{m+m'},y_0,\ldots,y_{m+m'}]$. Actually, by $\Tor$ rigidity, the latter conditions are verified if and only if \[\Tor^1\big(\SS/\mathcal{I}_{m},\SS/\mc{I}_{m'}\big)=\mathcal{I}_{m}\cap \mathcal{I}_{m'}/\mathcal{I}_{m}\cdot \mathcal{I}_{m'}=0.\]
To show this last condition, i.e.\ to show that \[\mathcal{I}_{m}\cap \mathcal{I}_{m'}=\mathcal{I}_{m}\cdot \mathcal{I}_{m'},\] we compute $\mathcal{I}_{m}\cap \mathcal{I}_{m'}$ by eliminating $t$ in the ideal $t\mathcal{I}_{m}+(1-t)\mathcal{I}_{m'}\subset\SS[t]$ see \cite[Prop. 2.3.5]{AL1994GB} for this standard use of Gr\"obner bases. Using the graded reverse lexicographic order, the variable $t$ being bigger than the $x$ and $y$ variables, and denoting
\begin{align*}
&\forall i\in\lbrace 0,\ldots,m-1\rbrace,\;g_i=t(x_iy_i-x_{i+1}y_{i+1}),\\
&\forall j\in\lbrace m,\ldots,m+m'\rbrace,\;g_j=(1-t)(x_jy_j-x_{j+1}y_{j+1}).
\end{align*}
and $\mathcal{G}=\lbrace g_0,\ldots,g_{m+m'}\rbrace$, a direct computation shows that for any $i\in\lbrace 0,\ldots,m-1\rbrace$ and $j\in\lbrace m,\ldots,m+m'\rbrace$ one has that $S(g_i,g_j)$ reduces to $(x_iy_i-x_{i+1}y_{i+1})(x_jy_j-x_{j+1}y_{j+1})\in \mc{I}_m\cdot \mc{I}_{m'}$ modulo $g_0,\ldots,g_{m+m'}$ (where $S(g_i,g_j)$ is the $S$-polynomial associated to $g_i$ and $g_j$). Since $S(g_i,g_{i'})$ reduces to $0$ modulo $g_0,\ldots,g_{m+m'}$ if $i,i'$ are both elements of $\lbrace 0,\ldots,m-1\rbrace$ or both elements $\lbrace m,\ldots,m+m'\rbrace$ and since $S(g_i,S(g_{i'},g_j))$ reduces to $0$ modulo $g_0,\ldots,g_{m+m'}, S(g_0,g_{m}),\ldots, S(g_{m-1},g_{m+m'-1})$, one has then that \[(t\mathcal{I}_{m}+(1-t)\mathcal{I}_{m'})\cap \SS\subset \mathcal{I}_{m}\cdot \mathcal{I}_{m'}\] which concludes the first step and shows \Cref{item2CrmStd}.

We shows that the ideal $\mathcal{I}_{m+m'}$ of $\p{}(\I_{\gm{\tau_{m}}{\tau_{m'}}})$ (which is generated by the entries $\begin{pmatrix}
y_0 & \ldots & y_{m+m'}
\end{pmatrix}\Phi_{\gm{\tau_m}{\tau_{m'}}}$) is prime again by computing Gr\"obner basis. More precisely, a direct computation shows that the set $\mathcal{H}=\lbrace h_0=x_0y_0-x_1y_1,\ldots,h_{m+m'}=x_{m+m'-1}y_{m+m'-1}-x_{m+m'}y_{m+m'}\rbrace$ is a Gr\"obner basis of $\mathcal{I}_{m+m'}$ (all $S$-polynomials in the $h_i$ reduces to $0$ modulo $h_0,\ldots,$ $h_{m+m'-1}$). We show that $\mathcal{I}_{m+m'}$ is prime by applying the primality test \cite[Algoritm 4.4.1]{AL1994GB} since, given $i\in \lbrace 0\ldots,m+m'-1\rbrace$, $h_ix_iy_i-x_{i+1}y_{i+1}$ is irreducible in $\kk'[x_i]$ where $\kk'$ is the quotient field of the ring $\kk[x_{i+1},\ldots,x_n,y_0,\ldots,y_n]/(h_{i+1},\ldots,h_{m+m'-1})$. 
\end{itemize}
The last equality \eqref{eqProjDegStdCrem} follows then from applying classical formulas between binomial numbers.
\end{proof}

Let us emphasize again that all the previous results are well known and could be summed up by the fact that the base ideal of the standard Cremona maps is of linear type (see for instance \cite[Subsection 2.1]{RussoSimis2001OnBirMap}). However the scheme of our proof of \Cref{prpAppliKunnethStdCrem}, could virtually be applied to more general situations. Even if it is at the moment out of our reach, let us present the kind of situations we have in mind and that we verified experimentally in a all the examples we considered.

\begin{cjc}\label{prpAppliKunnethGen}
Following \Cref{ntt} about glued maps, assume moreover that $\kk$ is algebraically closed, and let
\[
\begin{tikzpicture}[
style1/.style={
  matrix of math nodes,
  every node/.append style={text width=#1,align=center,minimum height=5ex},
  nodes in empty cells,
  left delimiter=(,
  right delimiter=),
  }
  ]
\matrix[style1=0.85cm] (1mat)
{
  &   \\
  &  \\
  &  \\
};
\draw[dashed]
  (1mat-2-2.west) -- (1mat-2-2.east);
\draw[dashed]
  (1mat-2-1.south west) -- (1mat-2-1.south east);
\draw[dashed]
  (1mat-1-1.north east) -- (1mat-3-1.south east);

\node
  at (1mat-1-1.south) {$\Phi_g$};
\node
  at (1mat-1-2) {$0_{m\times m'}$};
\node 
  at (1mat-3-1) {$0_{m'\times m}$};
\node
  at (1mat-3-2) {$\Phi_{g'}$};

\begin{scriptsize}
\draw[decoration={brace,raise=12pt},decorate]
  (1mat-1-2.north east) -- 
  node[right=15pt] {$m$} 
  (1mat-2-2.east);
\draw[decoration={brace,amplitude=5pt,raise=12pt},decorate]
  (1mat-2-2.east) -- 
  node[right=15pt] {$m'+1$} 
  (1mat-3-2.south east);
\draw[decoration={brace,raise=7pt},decorate]
  (1mat-1-1.north west) -- 
  node[above=8pt] {$m$} 
  (1mat-1-1.north east);
\draw[decoration={brace,raise=7pt},decorate]
  (1mat-1-2.north west) -- 
  node[above=8pt] {$m'$} 
  (1mat-1-2.north east);
  
\draw[decoration={mirror, brace,amplitude=5pt,raise=11pt},decorate]
  (1mat-1-1.north west) -- 
  node[left=14pt] {$m+1$} 
  (1mat-2-1.south west);
\draw[decoration={mirror, brace,raise=12pt},decorate]
  (1mat-3-1.north west) -- 
  node[left=14pt] {$m'$} 
  (1mat-3-1.south west);
 
\end{scriptsize}
\node [left=70pt]{$\Phi_{\gm{g}{g'}}=(\Phi_{ij})_{\genfrac{}{}{0pt}{}{0\leqslant i \leqslant m+m'+1}{1\leqslant j \leqslant m+m'}}=$};
\node[right=60pt] {$\in\RR^{(m+m'+1)\times (m+m') }$};
\end{tikzpicture}
\]
be the matrix defined by the following data:
\begin{itemize}
\item for any $j\in\lbrace 1,\ldots,m\rbrace$, let $k_j\geqslant 2$, $\lambda_{1,j},\ldots,\lambda_{k_j,j}\in\RR_m=\kk[x_0,\ldots,x_m]\subset\RR$ and:
\begin{itemize}
\item for $i\in\lbrace 1,\ldots,m+1\rbrace$, let $\Phi_{ij}\in |\lambda_{1,j},\ldots,\lambda_{k_j,j}|$ be a general linear combination of $\lambda_{1,j},\ldots,\lambda_{k_j,j}$,
\item for $i\in\lbrace m+2,\ldots,m+m'+1\rbrace$, $\Phi_{ij}=0$.
\end{itemize}

\item for any $j\in\lbrace m+1,\ldots,m+m'\rbrace$, let $k_j\geqslant 2$ and $\lambda_{1,j},\ldots,\lambda_{k_j,j}\in\RR_{m'}=\kk[x_m,\ldots,x_{m+m'}]\subset\RR$ and:
\begin{itemize}
\item for $i\in\lbrace 1,\ldots,m\rbrace$, $\Phi_{ij}=0$,
\item for $i\in\lbrace m+1,\ldots,m+m'+1\rbrace$, let $\Phi_{ij}\in |\lambda_{1,j},\ldots,\lambda_{k_j,j}|$ be a general linear combination of $\lambda_{1,j},\ldots,\lambda_{k_j,j}$.
\end{itemize}
\end{itemize}
Let also $\Phi_{g}=(\Phi_{ij})_{\genfrac{}{}{0pt}{}{0\leqslant i \leqslant m+1}{1\leqslant j \leqslant m}}\in\RR_m^{(m+1)\times m}$ and $\Phi_{g'}=(\Phi_{ij})_{\genfrac{}{}{0pt}{}{m+1\leqslant i \leqslant m+m'+1}{m+1\leqslant j \leqslant m+m'}}\in\RR_n^{(m'+1)\times m'}$ and $\gm{g}{g'}:\p{m+m'}\dashrightarrow\p{m+m'}$, $g:\p{m}\dashrightarrow\p{m}$ and $g':\p{m'}\dashrightarrow\p{m'}$ be the determinantal maps defined by $\Phi_{\gm{g}{g'}}$, $\Phi_g$ and $\Phi_{g'}$ (where $\p{m}=\Proj(\RR_m)\subset\Proj(\RR)=\p{m+m'}$ and $\p{m'}=\Proj(\RR_{m'})\subset\Proj(\RR)=\p{m+m'}$).

Then for all $i\in\lbrace 0,\ldots,m+m'\rbrace$:
\begin{equation}\label{eqKunnethCP}\d{\gm{g}{g'}}{k}=\underset{p=0}{\overset{k}{\sum}}\d{g}{p}\d{g'}{k-p}.\end{equation}
\end{cjc}

\section{Projective degrees vs mixed volumes of Koszul-determinantal maps}\label{SecMixVol}

For the rest of the paper, we let $\kk=\mathbb{C}$. Given a map $f:\p{n}_\kk\dashrightarrow \p{n}_\kk$ of base ideal $\I_f$, the inclusion $\Gamma_f\subset \p{}(\I_f)$ of the graph $\Gamma_f$ of $f$ in the $\Proj$ of the symmetric algebra of $\I_f$, see below \Cref{dfnDetMap} for more details, provides the following estimations:
\begin{equation}\label{ineqGP}
\forall k\in\lbrace 0,\ldots,n\rbrace,\hspace{2em} \d{f}{k}\leqslant\e{\p{}(\I_f)}{n-k,k}.
\end{equation}
Let us now present a specificity when computing the projective degrees of the Koszul-determinantal map $f$ we are now going to define. Recall \Cref{dfnDetMap} that, assuming that $f$ is determinantal, the presentation matrix $\Phi_f$ of $\I_f$ is called the Hilbert-Burch matrix of $f$. In this latter case, letting
\begin{equation*}
\begin{tikzcd}[row sep=0.8em,column sep=1em,minimum width=2em]
0\ar[r]& \underset{k=1}{\overset{n}{\oplus}}\msf{R}(-d_k)\ar[r, "{\Phi_f}"]& \msf{R}^{n+1} \ar[rrr,"{(f_0 \; \ldots \; f_n)}"] & & & \I_f(d) \ar{r}& 0
\end{tikzcd}
\end{equation*} be a graded free resolution of the base ideal of $f$, we call the positive integer $d_1,\ldots,d_n\geqslant 1$ the \emph{Hilbert-Burch degrees of $f$}.

Remark from the equations $(\msf{y}_0\;\ldots \; \msf{y}_n)\Phi_f$ defining the embedding of $\p{}(\I_f)$ in $\p{n}_\kk\times\p{n}_\kk$ that the conditions:
\begin{equation}\label{condMinors}
\forall k\in \lbrace 1,\ldots, n-1\rbrace,\;\codim \V\big(\I_k(\Phi_f)\big)\geq n+1-k
\end{equation} are the required conditions in order that $\p{}(\I_f)$ has codimension $n$ in which case $\p{}(\I_f)$ is a complete intersection and a minimal free resolution of its coordinate ring is the Koszul complex on the entries of $(\msf{y}_0\;\ldots \; \msf{y}_n)\Phi_f$.

\begin{dfn}[Koszul-determinantal map]\label{dfnKoszDetMap} 
 A \emph{Koszul-determinantal map} $f=(f_0:\ldots:f_n):\p{n}_\kk\dasharrow \p{n}_\kk$ is a map such that the following two conditions are verified:
\begin{itemize}
\item  a minimal presentation of $\I_f=(f_0,\ldots,f_n)$ reads: 
\begin{equation*}
\begin{tikzcd}[row sep=0.8em,column sep=1em,minimum width=2em]
0\ar[r]& \msf{R}^n\ar[r, "{\Phi_f}"]& \msf{R}^{n+1} \ar[rrr,"{(f_0 \; \ldots \; f_n)}"] & & & \I_f \ar{r}& 0.
\end{tikzcd}
\end{equation*}
\item the conditions \eqref{condMinors} are satisfied.
\end{itemize}
\end{dfn}

\begin{prp}\label{prpEstimMDet}
Let $f=(f_0:\ldots:f_n):\p{n}_\kk\dasharrow\p{n}_\kk$ be a Koszul-determinantal map of Hilbert-Burch degree $d_1,\ldots,d_n\geqslant 1$, then for all $k\in\lbrace 0,\ldots,n\rbrace$,
\begin{equation}\label{ineqGPdet}
\d{f}{k}\leqslant \sigma_{n-k,n}(d_1,\ldots,d_n)=\underset{\lbrace i_1,\ldots,i_{k}\rbrace \subset \lbrace 1,\ldots,n\rbrace}{\sum} d_{i_1}\ldots d_{i_k}
\end{equation} where $\sigma_{k,n}(u_1,\ldots,u_n)$ is the $k$-th symmetric polynomial in $n$ variables.
\end{prp}

\begin{proof}
Letting $\Phi_f$ be the $(n+1)\times n$ Hilbert-Burch matrix of $f$, the ideal $\J_1$ of $\p{}(\I_f)$ in $\p{n}_\kk\times\p{n}_\kk$ is generated by the $n$ entries $\phi_1,\ldots,\phi_n$ of the line matrix $(\msf{y}_0\;\ldots\;\msf{y}_n)\Phi_f$. Since conditions \eqref{condMinors} are moreover satisfied, $\p{}(\I_f)$ has codimension $n$ in $\p{n}_\kk\times\p{n}_\kk$ so $\p{}(\I_f)$ is a complete intersection and the Koszul complex on $\phi_1,\ldots,\phi_n\in\msf{S}$
\begin{equation*}
\begin{tikzcd}[row sep=0.8em,column sep=1em,minimum width=2em]
0\ar[r]& \msf{S}(-\underset{k=1}{\overset{n}{\sum}}d_k,-n)\ar[r]& \ldots \ar[r]& \overset{2}{\wedge}\big(\underset{k=1}{\overset{n}{\oplus}}\msf{S}(-d_k,-1)\big) \ar[r]& \underset{k=1}{\overset{n}{\oplus}}\msf{S}(-d_k,-1) \ar[r]& \msf{S}
\end{tikzcd}
\end{equation*} 
provides a minimal bi-graded free resolution of the coordinate ring of $\p{}(\I_f)$ from which we extract the component of total degree $n$ of $\Num_{\msf{S}/\J_1}(1-T_0,1-T_1)$ (see \cite[III.3.6]{eisenbud2000geo}):
\begin{equation*}
\forall k\in\lbrace 0,\ldots,n\rbrace,\hspace{0.5em}\e{\p{}(\I_f)}{n-k,k}=s_{n-k,n}(d_1,\ldots,d_n)=\underset{\lbrace i_1,\ldots,i_{k}\rbrace \subset \lbrace 1,\ldots,n\rbrace}{\sum} d_{i_1}\ldots d_{i_k}.
\end{equation*}
The conclusion of the proposition follows then from \eqref{ineqGP}.
\end{proof}

We point out that the conclusion of \Cref{prpEstimMDet} is classical and had been established in more general context of maps $f:\p{n}_\kk\dasharrow\p{n'}_\kk$ with $n\leqslant n'$, see for instance \cite[Theorem 5.7]{CidRuiz2021MixedMultProjDeg}. It explains however why, when interested in determinantal Cremona maps, one has also to consider base ideal not of linear type or else, if the base ideal is of linear type, the entries of the presentation matrix can only contain linear polynomials (case $d_1=\ldots=d_n=1$ of the previous proposition). In other words, one has to consider the kernel of the surjection $\Sym(\I_f)\twoheadrightarrow\R(\I_f)$ in order to understand the level of approximations of the inequalities \eqref{ineqGPdet}. As we previously explained in \Cref{SecKunneth}, this kernel is described by the ideals $\I_k(\Phi_f)$ for $k\in\lbrace 1,\ldots,n-1\rbrace$ of $k$-minors ideals of the presentation matrix $\Phi_f$ of $\I_f$ which, in our context of determinantal base ideal $\I_f$, have an expected codimension:
\begin{equation*}
\forall k\in \lbrace 1,\ldots, n\rbrace,\;\codim \V\big(\I_k(\Phi_f)\big)\geq n+2-k,
\end{equation*} ensuring, if these latter conditions are verified, that $\I_f$ is of linear type \cite[Subsection 2.1]{RussoSimis2001OnBirMap}. Hence the conditions \eqref{condMinors} are the first step to consider when interested in determinantal Cremona maps.

\subsection{Bernstein theorem's bound on the number of solutions of a zero dimensional polynomial system}\label{recallMV}
When interested in the topological degree $\d{f}{0}$ of a determinantal map $f=(f_0:\ldots:f_n):\p{n}_\kk\dasharrow\p{n}_\kk$, remark that two polynomial systems can be considered. An initial one, in the end the main object of this work, is defined by $f$ itself: given $\mbf{y}=(\msf{y}_0:\ldots:\msf{y}_n)\in\p{n}_\kk$ in the target space of $f$, one wants to find the pre-images $\mbf{x}\in\p{n}_\kk$ of $\mbf{y}$, i.e.\ wants to solve the polynomial system
\begin{equation}
f(\mbf{x})=\mbf{y}\Leftrightarrow\begin{cases}\msf{y}_0&=f_0(\msf{x}_0,\ldots:\msf{x}_n)\\
&\vdots \\
\msf{y}_n&=f_n(\msf{x}_0,\ldots:\msf{x}_n)
\end{cases}
\end{equation}
Expressing $\mbf{y}$ as the intersection of $n$ hyperplans in the target space of $f$ and assuming that $\mbf{y}$ is general enough, B\'ezout theorem then asserts that $\d{f}{0}=\#f^{-1}(\lbrace \mbf{y}\rbrace)\leqslant d^n$ where $d=\d{f}{n-1}$ is the common degree of the polynomials $f_i$ generating the base ideal $\I_f=(f_0,\ldots,f_n)$ of $f$. An intermediate polynomial system is moreover defined by a presentation matrix $\Phi_f$ of $\I_f$. Indeed, by definition of $\Phi_f$, one has $(f_0\;\ldots\;f_n)\Phi_f=0$ so the polynomial system
\begin{equation}\label{polySystPres}
(\msf{y}_0\;\ldots\; \msf{y}_n)\Phi_f=0
\end{equation}
whose number of equations is the number of columns of $\Phi_f$, contains by definition $f^{-1}(\lbrace \mbf{y}\rbrace)$. Now consider the situation where $\Phi_f=(\phi_{ij})_{\genfrac{}{}{0pt}{}{0\leqslant i \leqslant n}{ 1\leqslant j \leqslant n}}$ is a $(n+1)\times n$ matrix where each entry $\phi_{ij}$ of the $j$-th column of $\Phi_f$ has degree $d_j\geqslant 1$ and in which case $d=\d{f}{n-1}=d_1+\ldots+d_n$ by Hilbert-Burch theorem. Then, provided that $\codim \V\big(\I_k(\Phi_f)\big)\geq n+1-k$ for all $k\in \lbrace 1,\ldots, n\rbrace$, the system \eqref{polySystPres} is $0$-dimensional in $\p{n}_\kk$ and, by B\'ezout theorem, has $d_1\cdot \ldots\cdot d_n$ solutions. Moreover, in case the entries of $\Phi_f$ are sparse polynomials that all verify the same algebraic constraints on their coefficients, one can refine the bound on the number of solutions of \eqref{polySystPres} by using Bernstein's theorem on sparse polynomials. It then gives a combinatorial translation, via the computation of \emph{mixed volumes} associated to the polynomial entries of $\Phi_f$, to the problem of detecting determinantal Cremona maps of given Hilbert-Burch degree $(d_1,\ldots,d_n)$ among all determinantal maps of Hilbert--Burch degree $(d_1,\ldots,d_n)$.

This idea structures our approach of \Cref{prpMMGluedMap}, see \cite{Trung2005MixedMO} or \cite{LazarsfeldMustata2009ConvexBodies} for developments in more general contexts and for pointers to the related literature. We refer to \cite[Chapter 7]{coXLittleOShea2005UsingAlgGeo} for the background in convex geometry and all the material we use for the actual computations. Following \cite[7.4]{coXLittleOShea2005UsingAlgGeo}, a set $C\subset\mathbb{R}^n$ is \emph{convex} if it contains any segments between two points in $C$ and the \emph{convex hull} $\Conv(S)$ of a subset $S\subset \mathbb{R}^n$ is the smallest convex set containing $S$. A \emph{polytope} is the convex hull $\Conv(A)$ of a finite set $A\subset\mathbb{R}^n$ and the polytopes which are the convex hull of points with integer coordinates are called \emph{lattice polytopes}.

Given $\phi=\underset{\alpha\in\mathbb{N}^n}{\sum} c_{\alpha}\msf{x}^{\alpha}\in\kk[\msf{x}_1,\ldots,\msf{x}_n]$ where $\msf{x}^{\alpha}=\msf{x}_1^{\alpha_1}\ldots \msf{x}_n^{\alpha_n}$, the \emph{Newton polytope} of $\phi$, denoted $\NP(\phi)$, is the lattice polytope $\NP(\phi)=\Conv\lbrace \alpha\in\mathbb{N}^n,\;c_{\alpha}\neq 0\rbrace$.

A polytope $P\subset\mathbb{R}^n$ has an $n$-dimensional volume $\Vol_n(P)$. Given two polytopes $P,Q\subset\mathbb{R}^n$ and a real number $\lambda\geqslant 0$, the \emph{Minkowski sum} of $P$ and $Q$, denoted $P+Q$ is the set \[P+Q:=\lbrace p+q,\;p\in P,q\in Q\rbrace\] where $p+q$ denotes the usual vector sum in $\mathbb{R}^n$ and the $\lambda P$ stands for the polytope $\lbrace \lambda p,\;p\in P\rbrace$ where $\lambda p$ is the usual scalar multiplication in $\mathbb{R}^n$.

Given any collection $P_1,\ldots,P_r\subset\mathbb{R}^n$ and $r$ non negative scalar $\lambda_1,\ldots,\lambda_r\in\mathbb{R}$, then $\Vol_n(\lambda_1P_1+\ldots\lambda_rP_r)$ is a homogeneous polynomial of degree $n$ in the $\lambda_i$ \cite[7. Prop.4.9]{coXLittleOShea2005UsingAlgGeo}.
\begin{dfn}[mixed volume of a collection of polytopes]
The $n$-dimensional \emph{mixed volume} $\MV_n(P_1,\ldots,P_n)$ of given polytopes $P_1,\ldots,P_n$ is the coefficient of the monomial $\lambda_1,\ldots,\lambda_n$ in $\Vol_n(P_1,\ldots,P_n)$.
\end{dfn}
\begin{thm}\label{MVprop}\cite[7. Th.4.12]{coXLittleOShea2005UsingAlgGeo} \begin{enumerate}[label=(\roman*)]\item\label{MVtranslation} The mixed volume $\MV_n(P_1,\ldots,P_n)$ is invariant if the $P_i$ are replaced by their images under a volume-preserving transformation of $\mathbb{R}^n$.
\item\label{MVsymmetry} $\MV_n(P_1,\ldots,P_n)$ is symmetric and linear in each variable (multilinearity of the mixed volume).
\item\label{MVvanish} $\MV_n(P_1,\ldots,P_n)\geqslant 0$ and 
$\MV_n(P_1,\ldots,P_n)=0$ if one of the $P_i$ has dimension zero (i.e. if $P_i$ consists of a single point).

 $\MV_n(P_1,\ldots,P_n)>0$ if every $P_i$ has dimension $n$.
\item\label{MVcribleFormula}$\MV_n(P_1,\ldots,P_n)=\underset{k=1}{\overset{n}{\sum}}(-1)^{n-k}\underset{\genfrac{}{}{0pt}{}{I\subset\lbrace 1,\ldots,n\rbrace}{ |I|=k}}{\sum}\Vol_n(\underset{i\in I}{\sum}P_i)$ where $\underset{i\in I}{\sum}P_i$ is the Min\-kowski sum of polytope.
\end{enumerate}
\end{thm}

In addition to \Cref{MVprop}, our work relies also on the following toolbox.
\begin{lmm}\label{lmm6} \cite[Lemma 6]{SteffensTheobald2010MixedVolumesTechniques}
Let $n,n'\in\mathbb{N}^*$ and let $P_1,\ldots,P_n$ be polytopes in $\mathbb{R}^{n+n'}$ and $P_{n+1},\ldots,P_{n+n'}$ be polytopes in $\mathbb{R}^n \times\lbrace 0_{\mathbb{R}^{n'}}\rbrace\subset\mathbb{R}^{n+n'}$. Then:
\begin{equation*}
\MV_{n+n'}(P_1,\ldots,P_{n+n'})=\MV_n(P_1,\ldots,P_n)\MV_{n'}(\pi_n(P_{n+1}),\ldots,\pi_{n'}(P_{n+n'}))
\end{equation*}
where $\pi_{n'}:\mathbb{R}^{n+n'}\rightarrow \mathbb{R}^{n'}$ stands for the projection on the last $n'$ coordinates.
\end{lmm}

Following \cite[7 Section 5]{coXLittleOShea2005UsingAlgGeo}, let us now define the genericity of a polynomial with respect to a polytope.
\begin{dfn}[genericity with respect to a polytope]\label{dfnGenPoly}
Given finite set $A\subset\mathbb{Z}^n$, put $L(A):=\lbrace \underset{\alpha\in A_i}{\sum}c_{\alpha}\msf{x}^{\alpha}\in\kk[\msf{x}_1,\ldots,\msf{x}_n]\rbrace$ and remark that each $L(A)$ can be considered as an affine space $\kk^{\# A_i}$ with the coordinate $c_{\alpha}$ as coordinates.

A polynomial $\underset{\alpha\in A_i}{\sum}c_{\alpha}\msf{x}^{\alpha}$ is said to be \emph{generic with respect to $L(A)$} if its coefficients are generic in $L(A)$.

Given $k\geqslant 1$ and finite sets $A_1,\ldots,A_k\in\mathbb{Z}^n$, a property is said to \emph{hold generically} for polynomials $(\phi_1,\ldots,\phi_n)\in L(A_1)\times\ldots\times L(A_k)$ is there is a non zero polynomial in the coefficients of the $\phi_i$ such that the property holds for all $\phi_1,\ldots,\phi_n$ for which the polynomial is non vanishing, in particular if every $\phi_1,\ldots,\phi_n$ is generic with respect to its own polytope $L(A_1),\ldots, L(A_k)$.
\end{dfn}

\begin{thm}[Bernstein theorem]\label{thmBernstein} Given polynomials $\phi_1,\ldots,\phi_n\in\kk[\msf{x}_1,\ldots,\msf{x}_n]$ with finitely many common zeroes in $(\kk^*)^n$, let $P_i=\NP(\phi_i)$. Then the number of common zeroes in $(\mathbb{C}^*)^n$ is bounded above by the mixed volume $\MV_n(P_1,\ldots,P_n)$. Moreover if each $\phi_i$ is generic with respect to $P_i$, the number of common zero solutions is exactly $\MV_n(P_1,\ldots,P_n)$.
\end{thm}
See \cite[Proof of 7. Th. 5.4]{coXLittleOShea2005UsingAlgGeo} and the references therein for highlights about Bernstein theorem.

\subsection{Projective degrees of Koszul-determinantal maps defined by sparse polynomials}
Given a homogeneous polynomial $\phi\in\msf{R}=\kk[\msf{x}_0,\ldots,\msf{x}_n]$, the Newton polytope $\NP(\phi)\subset \mathbb{R}^n$ associated to $\phi$ is the Newton polytope of the de-homogenization of $\phi$ with respect to a variable $\msf{x}_i$ (omitted when irrelevant) and, reciprocally, given a polytope $P\subset\mathbb{R}^n$, we denote $\phi\in P$ for a homogeneous polynomial $\phi\in\msf{R}$ whose de-homogenization with respect to the  variable $\msf{x}_i$ is in $P$. Furthermore, we consider more general Newton polytopes $P\subset\mathbb{R}^n\times\mathbb{R}^n$ corresponding to bi-homogeneous polynomials $\phi\in\msf{S}=\msf{R}[\msf{y}_0,\ldots,\msf{y}_n]$, that is the bi-de-homogenization of $\phi$ with respect to one fixed variable $\msf{x}_i$ and one fixed variable $\msf{y}_j$ ($i,j\in\lbrace 0,\ldots,n\rbrace$) is in $P$.

Let us also denote by $S_n=\Conv\lbrace (\underbrace{0,\ldots,0}_{n}),(1,\underbrace{0,\ldots,0}_{n-1}),\ldots,(\underbrace{0,\ldots,0}_{n-1},1)\rbrace\subset\mathbb{R}^n$ the unit simplex of $\mathbb{R}^n$  and let $S_n^{\mbf{x}}:=S_n\times \lbrace 0_{\mathbb{R}^n}\rbrace\subset\mathbb{R}^n\times\mathbb{R}^n$ and $S_n^{\mbf{y}}:=\lbrace 0_{\mathbb{R}^n}\rbrace\times S_n\subset\mathbb{R}^n\times\mathbb{R}^n$. %Lastly, write $\pi_n:\mathbb{R}^n\times\mathbb{R}^n\rightarrow \mathbb{R}^n$ (resp. $\pi'_n:\mathbb{R}^n\times\mathbb{R}^n\rightarrow \mathbb{R}^n$) for the first (resp. second) projection.

\begin{prp}\label{prpRaffinementIneq}
Let $f:\p{n}_\kk\rightarrow\p{n}_\kk$ be a determinantal map of Hilbert-Burch matrix $\Phi_f=(\phi_{ij})_{\genfrac{}{}{0pt}{}{0\leqslant i \leqslant n}{ 1\leqslant j\leqslant n}}$ of Hilbert-Burch degree $(d_1,\ldots,d_n)$. Then for all $k\in\lbrace 0,\ldots,n\rbrace$:
\begin{equation*}
\d{f}{k}\leqslant \MV_{2n}(\underbrace{S_n^{\mbf{x}},\ldots,S_n^{\mbf{x}}}_{k},P_1^{\mbf{y}},\ldots,P_n^{\mbf{y}},\underbrace{S_n^{\mbf{y}},\ldots,S_n^{\mbf{y}}}_{n-k})\leqslant \sigma_{n-k,n}(d_1,\ldots,d_n)
\end{equation*}
where for $l\in\lbrace 1,\ldots,n\rbrace$, $P_l^{\mbf{y}}\subset\p{n}_\kk\times\p{n}_\kk$ is the Newton polytope of the $l$-th entry of the matrix $(y_0\;\ldots\;y_n)\Phi_f$ and $\sigma_{k,n}$ is the $k$-th symmetric polynomial in $n$-variables.
\end{prp}

\begin{proof}
As explained in \Cref{rmkCornerstone}, for all $k\in\lbrace 0,\ldots,n\rbrace$, $\e{\p{}(\I_f)}{n-k,k}$ is the number of common zero solutions in $\subset\p{n}_\kk\times\p{n}_\kk$ of the polynomials 
\begin{equation}\label{eqPolySystProjDeg} l_{1,0},\ldots,l_{k,0},\phi_1,\ldots,\phi_n,l_{0,1},\ldots,l_{0,n-k}
\end{equation}
where $l_{1,0},\ldots,l_{k,0}$ (resp. $l_{0,1},\ldots,l_{0,n-k}$) are generic with respect to $S_n^{\mbf{x}}$ (resp. generic with respect to $S_n^{\mbf{y}}$) and $\phi_1,\ldots,\phi_n\in\msf{S}=\msf{R}[\msf{y}_0,\ldots,\msf{y}_n]$ are the entries of the matrix $(\msf{y}_0\;\ldots\;\msf{y}_n)\Phi_f$ (the associated polynomial system being $0$-dimensional by definition of a determinantal map, see \Cref{dfnDetMap}).

In this setting, by \Cref{thmBernstein} in the generic case, the quantity \[\MV_{2n}(\underbrace{S_n^{\mbf{x}},\ldots,S_n^{\mbf{x}}}_{k},P_1^{\mbf{y}},\ldots,P_n^{\mbf{y}},\underbrace{S_n^{\mbf{y}},\ldots,S_n^{\mbf{y}}}_{n-k})\] is the number of solutions of \eqref{eqPolySystProjDeg} whose coordinates are all non zero which shows the right hand side inequality of \Cref{prpRaffinementIneq} since the number of common zero of \eqref{eqPolySystProjDeg} is equal to $\sigma_{n-k,n}(d_1,\ldots,d_n)$.

The set of common zero of \eqref{eqPolySystProjDeg} contains moreover the set \[H_\mbf{x}^{k} \cap \G_f\cap H_\mbf{y}^{n-k}\subset\p{n}_\kk\times\p{n}_\kk\] where $H_\mbf{x}^{k}=\V(l_{1,0},\ldots,l_{k,0})$ is the zero locus of $l_{1,0},\ldots,l_{k,0}$, $\G_f$ is the graph of $f$ and $H_\mbf{y}^{n-k}=\V(l_{0,1},\ldots,l_{0,n-k})$. However, by the genericity assumptions, remark that the points of $H_\mbf{x}^{k} \cap \G_f\cap H_\mbf{y}^{n-k}$ have all non zero coordinates which shows the left hand side inequality of \Cref{prpRaffinementIneq}.
\end{proof}

Let us now isolate two technical facts that rely on \Cref{lmm6}:
\begin{lmm}\label{lmmDegTopMV} Let $P_1,\ldots,P_n\subset\mathbb{R}^n$ be $n$ polytopes and for $j\in\lbrace 1,\ldots,n\rbrace$ put $P_j^{\mbf{y}}:=P_j\times\lbrace 0_{\mathbb{R}^n}\rbrace+\lbrace 0_{\mathbb{R}^n}\rbrace\times S_n^{\mbf{y}}\subset\mathbb{R}^n\times\mathbb{R}^n$.

Then \[\MV_{2n}(P_1^{\mbf{y}},\ldots,P_n^{\mbf{y}},\underbrace{S_n^{\mbf{y}},\ldots,S_n^{\mbf{y}}}_{n})=\MV_n(P_1,\ldots,P_n).\]
\end{lmm}
\begin{proof}
It suffices to apply the multilinearity of the mixed volume (\Cref{MVprop}. \Cref{MVsymmetry}), the projection formula (\Cref{lmm6}) and the vanishing condition (\Cref{MVprop}. \Cref{MVvanish}) on the left member of the equality to obtain:
\begin{align*}\MV_{2n}(P_1^{\mbf{y}},\ldots,P_n^{\mbf{y}},\underbrace{S_n^{\mbf{y}},\ldots,S_n^{\mbf{y}}}_{n})&=\MV_n(P_1,\ldots,P_n)\MV_n(\underbrace{S_n,\ldots,S_n}_{n})\\
&=\MV_n(P_1,\ldots,P_n).\end{align*}
\end{proof}

\begin{lmm}\label{lmmDecompoMV}
Let $m,m'\in\mathbb{N}^*$ and:

\begin{tabular}{c l l}
$\bullet$ & $P_1,\ldots,P_{m}$ & $\subset\mathbb{R}^{m}\times\lbrace 0_{\mathbb{R}^{m'}}\rbrace\subset\mathbb{R}^{m}\times \mathbb{R}^{m'}$,\\
$\bullet$ & $P_{m+1},\ldots,P_{m+m'}$ & $\subset\lbrace 0_{\mathbb{R}^{m}}\rbrace \times\mathbb{R}^{m'} \subset\mathbb{R}^{m}\times \mathbb{R}^{m'}$,
\end{tabular}

be $m+m'$ polytopes. For $j\in\lbrace 1,\ldots,m+m'\rbrace$ put:
\[P_j^{\mbf{y}}:=P_j\times\lbrace 0_{\mathbb{R}^{m+m'}}\rbrace+\lbrace 0_{\mathbb{R}^{m+m'}}\rbrace\times S_{m+m'}^{\mbf{y}}\subset\mathbb{R}^{m+m'}\times\mathbb{R}^{m+m'}.\] Then for all $k\in\lbrace 0,\ldots,m+m'\rbrace$,
\begin{gather*} \MV_{2(m+m')}(\underbrace{S_{m+m'}^{\mbf{x}},\ldots,S_{m+m'}^{\mbf{x}}}_{m+m'-k},P_1^{\mbf{y}},\ldots,P_{m+m'}^{\mbf{y}},\underbrace{S_{m+m'}^{\mbf{y}},\ldots,S_{m+m'}^{\mbf{y}}}_{k})\\
= \\
\underset{p=0}{\overset{k}{\sum}}\big[ \big( \underset{\lbrace l_1,\ldots,l_p\rbrace \subset \atop\lbrace 1,\ldots,m'\rbrace}{\sum} \MV_{m'}(\underbrace{S_{m'},\ldots,S_{m'}}_{m'-p},P_{l_1},\ldots,P_{l_p})\big) \times \\
\big( \underset{\lbrace l_1,\ldots,l_{k-p}\rbrace  \subset \atop\lbrace m'+1,\ldots,m+m'\rbrace}{\sum} \MV_m(\underbrace{S_{m},\ldots,S_{m}}_{m+p-k},P_{l_1},\ldots,P_{l_{k-p}}) \big) \big]
\end{gather*}
\end{lmm}

\begin{proof}
The formula follows from by decomposing first $S_{m+m'}^{\mbf{x}}\subset\mathbb{R}^{m+m'}\times \mathbb{R}^{m+m'}$ (resp. $S_{m+m'}^{\mbf{y}}$) as the sum $S_{m}^{\mbf{x}}\times\lbrace 0_{\mathbb{R}^{m'}}\rbrace\times\lbrace 0_{\mathbb{R}^{m+m'}}\rbrace+\lbrace 0_{\mathbb{R}^{m}}\rbrace\times S_{m'}^{\mbf{x}}\times\lbrace 0_{\mathbb{R}^{m+m'}}\rbrace\subset\mathbb{R}^{m+m'}\times\mathbb{R}^{m+m'}$ (resp. as the sum $\lbrace 0_{\mathbb{R}^{m+m'}}\rbrace\times S_{m}^{\mbf{y}}\times\lbrace 0_{\mathbb{R}^{m'}}\rbrace+\lbrace 0_{\mathbb{R}^{m+m'}}\rbrace\times\lbrace 0_{\mathbb{R}^{m}}\rbrace\times S_{m'}^{\mbf{y}}\times\lbrace 0_{\mathbb{R}^{m+m'}}\rbrace$) and then applying the multilinearity of the mixed volume (\Cref{MVprop}. \Cref{MVsymmetry}), the projection formula (\Cref{lmm6}), the vanishing condition (\Cref{MVprop}. \Cref{MVvanish}) and eventually \Cref{lmmDegTopMV}.
\end{proof}

\begin{prp}\label{prpEqMMMV}
Let $f:\p{n}_\kk\dasharrow\p{n}_\kk$ be a Koszul-determinantal map of Hilbert-Burch matrix $\Phi_f=(\phi_{ij})_{\genfrac{}{}{0pt}{}{0\leqslant i \leqslant n}{ 1\leqslant j\leqslant n}}$. Then:
\begin{align*}\forall k\in\lbrace 0,\ldots,n\rbrace,\;\d{f}{k}= \MV_{2n}&(\underbrace{S_n^{\mbf{x}},\ldots,S_n^{\mbf{x}}}_{k},P_1^{\mbf{y}},\ldots,P_n^{\mbf{y}},\underbrace{S_n^{\mbf{y}},\ldots,S_n^{\mbf{y}}}_{n-k})\\
&\Updownarrow\\
\overline{\p{}(\I_f)\backslash\G_f}\subset\V&(\underset{i=0}{\overset{n}{\Pi}}\msf{x}_i)\subset\p{n}_\kk\times\p{n}_\kk
\end{align*}
where for $l\in\lbrace 1,\ldots,n\rbrace$, $P_l^{\mbf{y}}\subset\mathbb{R}^{n}\times\mathbb{R}^{n}$ is the Newton polytope of the $l$-th entry of the matrix $(\msf{y}_0\;\ldots\;\msf{y}_n)\Phi_f$.
\end{prp}

\begin{proof}
Assume that $\overline{\p{}(\I_f)\backslash\G_f}\subset\V(\underset{i=0}{\overset{n}{\Pi}}\msf{x}_i)\subset\p{n}_\kk\times\p{n}_\kk$ so $\p{}(\I_f)$ and $\G_f$ coincide away the coordinate exis defined by the $x_i$. Now for any  $k\in\lbrace 0,\ldots,n\rbrace$,
\begin{equation}\label{eqMVAux}
MV_{2n}(\underbrace{S_n^{\mbf{x}},\ldots,S_n^{\mbf{x}}}_{k},P_1^{\mbf{y}},\ldots,P_n^{\mbf{y}},\underbrace{S_n^{\mbf{y}},\ldots,S_n^{\mbf{y}}}_{n-k})
\end{equation} is the number of point with non zero coordinate in the $\mbf{x}$-variable of the reduced scheme $H_\mbf{x}^{k}\cap \p{}(\I_f)\cap H_\mbf{y}^{n-k})$
where $H_\mbf{x}^{k}$ is the zero locus of $k$ general linear forms in the $\mbf{x}$-variables and $H_\mbf{y}^{n-k}$ is the zero locus of $c-k$ general linear forms in the $\mbf{y}$-variables. So \eqref{eqMVAux} is also equal to the number of point with non zero coordinate in the $\mbf{x}$-variable of the reduced scheme $H_\mbf{x}^{k}\cap \G_f\cap H_\mbf{y}^{n-k})$. But since $H_\mbf{x}^{k}$ and $H_\mbf{y}^{n-k}$ are general linear subspace, this last quantity is equal to $\d{f}{k}$ which show the indirect implication in \Cref{prpEqMMMV}.

Now assume that $\overline{\p{}(\I_f)\backslash\G_f}\not\subset\V(\underset{i=0}{\overset{n}{\Pi}}\msf{x}_i)$. This implies that there exists a codimension $n$ irreducible component $\mathbb{T}$ of $\p{}(\I_f)$, distinct from $\G_f$, that is not included in $\V(\underset{i=0}{\overset{n}{\Pi}}\msf{x}_i)$. Since $\mathbb{T}$ lies above the zero locus of some Fitting ideals of $\Phi_f$, it not included neither in $\V(\underset{i=0}{\overset{n}{\Pi}}\msf{y}_i)$. Hence, there is a $k\in\lbrace 0,\ldots,n \rbrace$ such that for any $H_\mbf{x}^{k}$ which is the zero locus of $k$ general linear forms in the $\mbf{x}$-variables and any $H_\mbf{y}^{k}$ which is the zero locus of $c-k$ general linear forms in the $\mbf{y}$-variables, the scheme $H_\mbf{x}^{k}\cap \mathbb{T}\cap H_\mbf{y}^{n-k}$ contain a point with all non zero coordinate. Since this point is taken in account in the mixed volume \eqref{eqMVAux}, one has the strict inequality $\d{f}{k}<\eqref{eqMVAux}$.
\end{proof}
Hence, to know the actual term $\d{f}{k}$ of the projective degrees of a determinantal map $f$, the computation of the associated mixed volume has to be completed by a preliminary control on the support of the successive ideal of minors of $\Phi_f$. We illustrate such a control in the next subsection.

\subsection{Gluing determinantal maps subject to general conditions}\label{SecGluedMap}
Let us now focus on another construction of a \emph{glued determinantal map} \[\gm{g}{g'}:\p{m+m'}_\kk\dasharrow\p{m+m'}_\kk,\] where $m,m'\in\mathbb{N}^*$, starting from two determinantal Cremona maps $g:\p{m}_\kk\dasharrow\p{m}_\kk$ and $g':\p{m'}_\kk\dasharrow\p{m'}_\kk$. In the following, given polynomials $\phi,\psi_1,\ldots,\psi_{l}$, we denote by $\phi\in |\psi_1,\ldots,\psi_{l}|$ the condition that $\phi =\underset{k=1}{\overset{l}{\sum}}\lambda_k\psi_k$ is a general linear combination of $\psi_1,\ldots,\psi_{l}$.

\begin{prpdfn}[glued map]\label{prpdfnGluedMap}

Given any $j\in\lbrace 1,\ldots,m+m'\rbrace$, let $l_j\in\mathbb{N}$ and let:
\begin{itemize}
\item $\psi_0^{(j)},\ldots,\psi_{l_j}^{(j)}\in\RR_m=\kk[x_0,\ldots,x_m]$ if $j\in\lbrace 1,\ldots,m\rbrace$
\item $\psi_0^{(j)},\ldots,\psi_{l_j}^{(j)}\in\RR_{m'}=\kk[x_{m},\ldots,x_{m+m'}]$ if $j\in\lbrace m+1,\ldots,m+m'\rbrace$.
\end{itemize}
Let also $\Phi_g=(\phi_{ij}^{(g)})_{\genfrac{}{}{0pt}{}{0\leqslant i \leqslant m}{ 1\leqslant j \leqslant m}}\in\RR_m^{(m+1)\times m}$, $\Phi_{g'}=(\phi_{ij}^{(g')})_{\genfrac{}{}{0pt}{}{0\leqslant i \leqslant n}{ m+1\leqslant j \leqslant m+m'}}\in\RR_{m'}^{(m'+1)\times m'}$ and $\Phi_{\gm{g}{g'}}=(\phi_{ij})_{\genfrac{}{}{0pt}{}{0\leqslant i \leqslant m+m'}{ 1\leqslant j \leqslant m+m'}}\in\RR_{m+m'}^{(m+m'+1)\times (m+m')}$ where $\RR_{m+m'}=\kk[x_0,\ldots,x_{m+m'}]$ be such that:
\begin{itemize}
\item for any $j\in\lbrace 1,\ldots, m\rbrace$, each entry $\phi_{ij}^{(g)}$ of the $j$-th column of $\Phi_g$ is such that $\phi_{ij}^{(g)}\in|\psi_0^{(j)},\ldots,\psi_{l_j}^{(j)}|$.
\item for any $j\in\lbrace 1,\ldots, m'\rbrace$, each entry $\phi_{ij}^{(g')}$ of the $j$-th column of $\Phi_{g'}$ is such that $\phi_{ij}^{(g')}\in |\psi_0^{(j+m)},\ldots,\psi_{l_j+m}^{(j+m)}|$.
\item for any $j\in\lbrace 1,\ldots, m+m'\rbrace$, each entry $\phi_{ij}$ of the $j$-th column of $\Phi_{\gm{g}{g'}}$ is such that $\phi_{ij}\in |\psi_0^{(j)},\ldots,\psi_{l_j}^{(j)}|$.
\end{itemize}
Assume that $\codim \V\big(\I_m(\Phi_g)\big)=\codim \V\big(\I_{m'}(\Phi_{g'})\big)=2$, then \[\codim \V\big(\I_{m+m'}(\Phi_{\gm{g}{g'}})\big)=2\] and define the \emph{glued map} $\gm{g}{g'}:\p{m+m'}_\kk\dasharrow\p{m+m'}_\kk$ as the map whose base locus $\I_{\gm{g}{g'}}$ is the $(m+m')$-minors ideal of $\Phi_{\gm{g}{g'}}$.
\end{prpdfn}
In the following, under the notation of \Cref{prpdfnGluedMap} and under the assumption that $\codim \V\big(\I_m(\Phi_g)\big)=\codim \V\big(\I_{m'}(\Phi_{g'})\big)=2$, we let $g:\p{m}_\kk\dasharrow\p{m}_\kk$ (resp. $g':\p{m'}_\kk\dasharrow\p{m'}_\kk$) be the map whose base locus $\I_g$ is equal to $\I_m(\Phi_g)$ (resp. $\I_{g'}=\I_m(\Phi_{g'})$).

\begin{proof}
Assume that $\codim \V\big(\I_m(\Phi_g)\big)=\codim \V\big(\I_{m'}(\Phi_{g'})\big)=2$ and suppose by contradiction that $\codim \V\big(\I_{m+m'}(\Phi_{\gm{g}{g'}})\big)<2$. Then there is a common factor to each $m+m'$ minors of $\Phi_{\gm{g}{g'}}$, so, after operation on columns, one column of $\Phi_{\gm{g}{g'}}$ have all its entries sharing a common factor. But it is impossible under the genericity assumption on the entries of $\Phi_{\gm{g}{g'}}$.
\end{proof}

In the following example, we illustrate however that the glued map $\gm{g}{g'}$ of two Koszul-determinantal maps $g$ and $g'$, though determinantal by \Cref{prpdfnGluedMap}, may not be Koszul-determinantal.

\begin{ex}
Put $m=n=2$ and let $l_1=l_2=l_3=l_4=2$ with:
\begin{itemize}
\item $\psi_0^{(1)}=\psi_0^{(2)}=x_1$, $\psi_1^{(1)}=\psi_1^{(2)}=x_2$
\item $\psi_0^{(3)}=\psi_0^{(4)}=x_2$, $\psi_1^{(3)}=\psi_1^{(4)}=x_3$.
\end{itemize}
Then, the maps $g:\p{2}_\kk\dasharrow\p{2}_\kk$ and $g':\p{2}_\kk\dasharrow\p{2}_\kk$ are Koszul-determinantal but $\gm{g}{g'}:\p{4}_\kk\dasharrow\p{4}_\kk$ is such that $\codim \V\big(\I_{1}(\Phi_{\gm{g}{g'}})\big)=3<4$ ($\I_{1}(\Phi_{\gm{g}{g'}})=(x_1,x_2,x_3)$).
\end{ex}

Let us now describe the projective degrees of a glued determinantal map in the almost linear setting.
\begin{prp}\label{prpMMGluedMap}
Let $m,d\geqslant 1$, put $l_1=\ldots=l_m=m+1,l_{m+1}=1,l_{m+2}=m+3$
\begin{itemize}
\item $\forall j\in\lbrace 1,\ldots, m\rbrace, k\in\lbrace 0,\ldots, m\rbrace$, let $\psi_k^{(j)}=x_k$
\item $\psi_0^{(m+1)}=x_0$, $\psi_1^{(m+1)}=x_1$
\item $\forall k\in\lbrace 1,\ldots,3d\rbrace$, $\psi_k^{(m+2)}$ is the $k$-th generator of the product  \[(x_m,x_{m+1},x_{m+2})\cdot(x_{m+1},x_{m+2})^{d-1}=(x_mx_{m+1}^{d-1},x_mx_{m+2}^{d-1})+(x_{m+1},x_{m+2})^{d}.\]
\end{itemize}
Let $\gm{g}{g'}:\p{m+2}_\kk\dasharrow\p{m+2}_\kk$ be the glued map associated to the previous data as defined in \Cref{prpdfnGluedMap} and whose base ideal $\I_{\gm{g}{g'}}$ is the $(m+2)$-minors ideal of $\Phi_{\gm{g}{g'}}$.

Then $\gm{g}{g'}$ is a determinantal map and moreover:
\[\forall k\in \lbrace 0,\ldots,m+2\rbrace,\,\d{\gm{g}{g'}}{k}=\binom{m}{m-k}+(d+1)\binom{m}{m-k+1}+\binom{m}{m-k+2}\] with the convention that $\binom{j}{i}=0$ if $i<0$ or $i>j$.
\end{prp}

\begin{proof}
First, under our general assumptions, $\codim\I_{m+2}(\Phi_{\gm{g}{g'}})=2$ and for all $k\in \lbrace 1,\ldots,n\rbrace$, $\codim\I_{k}(\Phi_{\gm{g}{g'}})\geqslant 3+m-k$, hence $\gm{g}{g'}$ is Koszul-determinantal. Moreover $\overline{\p{}(\I_{\gm{g}{g'}})\backslash\G_{\gm{g}{g'}}}\subset\V(\underset{i=0}{\overset{n}{\Pi}}x_i)\subset\p{m+2}_\kk\times\p{m+2}_\kk$ hence, by \Cref{prpEqMMMV}, we can use the mixed volumes of the polytopes $P_j=\NP(\psi_{1}^{(j)})+\ldots+\NP(\psi_{l_j}^{(j)})$ ($j\in\lbrace 1,\ldots,m+2\rbrace$) to compute the projective degrees of $\gm{g}{g'}$.

By applying \Cref{lmmDecompoMV}, given any $k\in \lbrace 0,\ldots,m+2\rbrace$, one has the formula:
\begin{gather*} \MV_{2(m+2)}(\underbrace{S_{m+2}^{\mbf{x}},\ldots,S_{m+2}^{\mbf{x}}}_{m+2-k},P_1^{\mbf{y}},\ldots,P_{m+2}^{\mbf{y}},\underbrace{S_{m+2}^{\mbf{y}},\ldots,S_{m+2}^{\mbf{y}}}_{k})\\
= \\
\underset{p=0}{\overset{k}{\sum}}\big[ \big( \underset{\lbrace l_1,\ldots,l_{k-p}\rbrace  \subset \atop\lbrace 1,\ldots,m\rbrace}{\sum} \MV_m(\underbrace{S_{m},\ldots,S_{m}}_{m+p-k},P_{l_1},\ldots,P_{l_{k-p}}) \big)\times \\
 \big( \underset{\lbrace l_1,\ldots,l_p\rbrace \subset \atop\lbrace m+1,m+2\rbrace}{\sum} \MV_2(\underbrace{S_{2},\ldots,S_{2}}_{2-p},P_{l_1},\ldots,P_{l_p})\big)
 \big] \\
\Updownarrow\\
\d{g|g'}{k}=\underset{p=0}{\overset{k}{\sum}}\big[  \d{g}{p}\d{g'}{k-p} \big]
\end{gather*}
with the convention that $\d{g}{p}=0$ if $p>m$ and $\d{g'}{p}=0$ if $p>2$.

The result of \Cref{prpMMGluedMap} follows from the fact that:
\begin{itemize}
\item $\d{g}{}=(1,\binom{m}{1},\ldots,\binom{m}{m-1},\binom{m}{m})$ as $g$ is general determinantal map, see \cite[Theorem 2]{PanGonzalez2005CarClassesDetCremTrans}
\item $\d{g'}{}=(\d{g'}{0},\d{g'}{1},\d{g'}{2})=(1,d+1,1)$, this follows for instance from \cite[Th. 5.14]{BuCiDAnd2018MultiGrad}.
\end{itemize}
\end{proof}

\bibliographystyle{alpha}
%\bibliography{../bibliographie}

\end{document}